\documentclass{article}

\usepackage[T1]{fontenc}
\usepackage[utf8]{inputenc}
\usepackage[french,english]{babel}

\usepackage{amsmath,amssymb,amsthm}
\usepackage{stmaryrd}

\usepackage{xcolor,hyperref, graphicx}
\hypersetup{
    colorlinks,
    linkcolor={red!50!black}, 
    citecolor={green!40!black},
    urlcolor={blue!80!black}
}
\usepackage[capitalize]{cleveref}

\newcommand{\Fund}{\mathcal{F}}
\newcommand{\R}{\mathbb{R}}
\newcommand{\C}{\mathbb{C}}

\newcommand{\Z}{\mathbb{Z}}
\newcommand{\Q}{\mathbb{Q}}
\newcommand{\N}{\mathbb{N}}
\newcommand{\Half}{\mathcal{H}}
\newcommand{\eps}{\varepsilon}

\newcommand{\mat}[4]{\left(\begin{matrix}#1&#2\\#3&#4\end{matrix}\right)}
\newcommand{\vect}[2]{\left(\begin{matrix}#1\\#2\end{matrix}\right)}
\newcommand{\svec}[2]{\left(\begin{smallmatrix}#1\\#2\end{smallmatrix}\right)}
\newcommand{\abs}[1]{\left|#1\right|}
\newcommand{\norm}[1]{\left\lVert#1\right\rVert}

\newcommand{\from}{\colon}
\newcommand{\Zint}[1]{\left\llbracket #1\right\rrbracket}
\newcommand{\babs}[1]{\bigl|#1\bigr|}
\newcommand{\bbabs}[1]{\biggl|#1\biggr|}

\newcommand{\overbar}[1]{\mkern
  1.5mu\overline{\mkern-1.5mu#1\mkern-1.5mu}\mkern 1.5mu}
\newcommand{\conj}[1]{\overbar{#1}}

\DeclareMathOperator{\im}{Im}
\DeclareMathOperator{\re}{Re}
\DeclareMathOperator{\Sp}{Sp}
\DeclareMathOperator{\GSp}{GSp}
\DeclareMathOperator{\Diag}{Diag}
\DeclareMathOperator{\Tr}{Tr}
\DeclareMathOperator{\SL}{SL}
\DeclareMathOperator{\GL}{GL}

\newtheorem{prop}{Proposition}[section]
\Crefname{prop}{Proposition}{Propositions}
\newtheorem{thm}[prop]{Theorem}
\newtheorem{lem}[prop]{Lemma}
\Crefname{lem}{Lemma}{Lemmas}
\newtheorem{cor}[prop]{Corollary}
\Crefname{item}{Item}{Items}
\theoremstyle{definition}
\newtheorem{algo}[prop]{Algorithm}

\newcommand{\addpic}[1]{\begin{center}\includegraphics[scale=0.5]{#1}\vspace{5mm}\end{center}}

\title{Sign choices in the AGM for genus two theta constants}
\author{Jean Kieffer}

\begin{document}

\selectlanguage{english}

\maketitle

\begin{abstract}
  Existing algorithms to compute genus~$2$ theta constants in
  quasi-linear time use Borchardt sequences, an analogue of the
  arithmetic-geometric mean for four complex numbers. In this paper,
  we show that these Borchardt sequences are only given by good
  choices of square roots, as in the genus~$1$ case. This removes the
  sign indeterminacies when computing genus~$2$ theta constants
  without relying on numerical integration.
\end{abstract}

\selectlanguage{french}

\begin{abstract}
  Les algorithmes existants pour le calcul de thêta-constantes en
  genre~$2$ en temps quasilinéaire utilisent des suites de Borchardt,
  un analogue de la moyenne arithmético-géométrique pour quatre
  nombres complexes. Dans cet article, nous montrons que ces suites de
  Borchardt sont constituées uniquement de bons choix de signes, comme
  c'est le cas en genre~$1$. Ce résultat permet de lever les
  indéterminations de signes lors du calcul de thêta-constantes en
  genre~$2$ sans recours à l'intégration numérique.
\end{abstract}

\selectlanguage{english}



\section{Introduction}

Denote by~$\Half_g$ the Siegel half space of principally polarized
abelian varieties of dimension~$g$, consisting of all matrices
$\tau\in M_g(\C)$ such that~$\tau$ is symmetric and $\im(\tau)$ is
positive definite; for instance, $\Half_1$ is the usual upper half
plane. The \emph{theta constants} are the holomorphic functions
on~$\Half_g$ defined by
\begin{equation}
  \label{eq:theta}
  \theta_{a,b}(\tau) = \sum_{m\in\Z^g} \exp\left(i\pi \left(\left(m+\frac a2\right)^t
      \tau\left(m+\frac a2\right) + \left(m+\frac a2\right)^t b\right)\right),
\end{equation}
where~$a$ and~$b$ run through $\{0,1\}^g$ (by convention, vectors in
formula~\eqref{eq:theta} are written vertically). Theta constants have
a fundamental importance in the theory of Siegel modular forms, as
every scalar-valued Siegel modular function of any weight on~$\Half_g$
has an expression in terms of quotients of theta
constants~\cite[Thm.~9 p.~222]{igusa_ThetaFunctions1972}.  Moreover,
for~$1\leq g\leq 3$, then the stronger result that every Siegel
modular form is a polynomial in the theta constants
holds~\cite{igusa_GradedRingThetaconstants1964,
  igusa_GradedRingThetaconstants1966,
  freitag_VarietyAssociatedRing2019}.

In numerical algorithms manipulating modular forms, the following
operations are therefore very common: first, given (quotients of)
theta constants at a given $\tau\in\Half_g$, compute~$\tau$\,; second,
given $\tau\in\Half_g$, compute the theta
constants~$\theta_{a,b}(\tau)$. For instance, these operations are
important building blocks in algorithms computing modular
polynomials~\cite{enge_ComputingModularPolynomials2009,
  milio_QuasilinearTimeAlgorithm2015,
  milio_ModularPolynomialsHilbert2020} or Hilbert class
polynomials~\cite{enge_ComplexityClassPolynomial2009,
  enge_ComputingClassPolynomials2014, streng_ComputingIgusaClass2014}
via complex approximations.

The arithmetic-geometric mean
(AGM)~\cite{borchardt_TheorieArithmetischgeometrischesMittels1888,
  cox_ArithmeticgeometricMeanGauss1984,
  bost_MoyenneArithmeticogeometriquePeriodes1988,
  jarvis_HigherGenusArithmeticgeometric2008} gives an algorithm to
find~$\tau$ given its theta constants. This algorithm is quasi-linear
in terms of the required precision. In order to compute theta
constants in quasi-linear time as well, a well-studied strategy is to
combine the AGM with Newton iterations. This strategy was first
described in~\cite{dupont_FastEvaluationModular2011} in the genus~$1$
case, in~\cite{dupont_MoyenneArithmeticogeometriqueSuites2006} in the
genus~$2$ case, and later extended to theta \emph{functions}, in
opposition to theta \emph{constants},
in~\cite{labrande_ComputingJacobiTheta2018,labrande_ComputingThetaFunctions2016}. These
references also outline extensions to higher genus.

\paragraph{The genus~1 case.}
Let us detail the genus~$1$ case to convey the general idea. After
reducing the argument~$\tau\in \Half_1$ using Gauss's
algorithm~\cite[§6.1]{streng_ComputingIgusaClass2014}, we can assume
that~$\tau$ belongs to the classical fundamental domain under the
action of~$\SL_2(\Z)$, denoted by~$\Fund_1$.

First assume that theta quotients at~$\tau\in\Fund_1$ are given. Then
the sequence
\begin{displaymath}
  B(\tau) = \left( \frac{\theta_{0,0}^2(2^n\tau)}{\theta_{0,0}^2(\tau)}, \,
    \frac{\theta_{0,1}^2(2^n\tau)}{\theta_{0,0}^2(\tau)} \right)_{n\geq 0}
\end{displaymath}
is an \emph{AGM sequence}, meaning that each term is obtained from
the previous one by means of the transformation
\begin{displaymath}
  (x,y) \mapsto \left(\frac{x+y}{2},\sqrt{x}\sqrt{y}\right)
\end{displaymath}
for some choice of the square roots. This is a consequence of the
duplication formula \cite[p.~221]{mumford_TataLecturesTheta1983}, the
correct square roots being the theta quotients themselves. In the
algorithm, the sign ambiguity is easily removed using the fact
that~$\sqrt{x}$ and~$\sqrt{y}$ should lie in a common open quarter
plane \cite[Thm.~2]{dupont_FastEvaluationModular2011}: we say that the
sequence~$B(\tau)$ is given by \emph{good sign choices}. It converges
quadratically to~$1/\theta_{0,0}^2(\tau)$, as the series expansion
\eqref{eq:theta} shows.

It turns out that the sequence~$B(-1/\tau)$ is also an AGM sequence
with good sign choices
\cite[Prop.~7]{dupont_FastEvaluationModular2011}. Its first term can
be computed from theta quotients at~$\tau$ using the transformation
formulas for theta constants under~$\SL_2(\Z)$. The limit
of~$B(-1/\tau)$ is $1/\theta_{0,0}^2(-1/\tau)$. Finally, we can
recover~$\tau$ using the formula
\begin{equation}
  \label{eq:tau-g1}
  \theta_{0,0}^2\left(\frac{-1}{\tau}\right) = -i\tau \theta_{0,0}^2(\tau).
\end{equation}
Since AGM sequences with good sign choices converge quadratically,
this gives an algorithm to \emph{invert} theta functions on~$\Fund_1$
with quasi-linear complexity in the output precision, at least for
fixed~$\tau$. This method was already known to Gauss \cite[X.1,
pp.\,184--206]{gauss_Werke1868}, and we recommend
\cite[§3C]{cox_ArithmeticgeometricMeanGauss1984} for a historical
exposition of Gauss's works on the AGM and elliptic functions.

In order to compute theta functions at a given~$\tau\in\Fund_1$, the
most efficient known method is to build a Newton
scheme~\cite{dupont_FastEvaluationModular2011}, using the AGM method
to invert theta constants. This yields a quasi-linear algorithm to
compute genus~$1$ theta constants, whose complexity can be made
uniform in~$\tau\in\Fund_1$
\cite[Thm.~5]{dupont_FastEvaluationModular2011}.

\paragraph{The genus~2 case.}
A similar strategy can be applied to theta functions in genus~$2$,
using \emph{Borchardt sequences}, a generalization of AGM sequences
for four complex
numbers~\cite{borchardt_TheorieArithmetischgeometrischesMittels1888,
  bost_MoyenneArithmeticogeometriquePeriodes1988,
  jarvis_HigherGenusArithmeticgeometric2008}. Let us refer
to~§\ref{sec:borchardt} for the definition of Borchardt sequences, the
numbering of genus~$2$ theta constants, and the definition of the
matrices $\gamma_k\in\Sp_4(\Z)$ for $0\leq k\leq 3$. The Borchardt
sequences we consider are the sequences $B(\gamma_k\tau)$ for
$0\leq k\leq 3$, where
\begin{displaymath}
  B(\tau) = \left(\frac{\theta_0^2(2^n\tau)}{\theta_0^2(\tau)},\,
    \frac{\theta_1^2(2^n\tau)}{\theta_0^2(\tau)},\,
    \frac{\theta_2^2(2^n\tau)}{\theta_0^2(\tau)},\,
    \frac{\theta_3^2(2^n\tau)}{\theta_0^2(\tau)} \right)_{n\geq 0}
\end{displaymath}
for every $\tau\in\Half_2$. Their first terms are given by different
combinations of theta quotients at~$\tau$ (see
\cref{cor:theta-transf}). It is known that for a given~$\tau$, all but
a finite number of sign choices in these Borchardt sequences are good,
and the other sign choices can be determined using certified
computations of hyperelliptic integrals at relatively low precision:
see the discussion before Prop.~3.3 in
\cite{labrande_ComputingThetaFunctions2016}, and
\cite{molin_ComputingPeriodMatrices2019} for an algorithm that
provides this input. However, the required precision and the cost of
the numerical integration algorithms depend heavily on~$\tau$.

Actually, when~$\tau$ belongs to the usual fundamental
domain~$\Fund_2$ under the action of~$\Sp_4(\Z)$, practical
experiments suggest that \emph{all} sign choices are good in the
genus~$2$ algorithm as well
\cite[Conj.~9.1]{dupont_MoyenneArithmeticogeometriqueSuites2006},
\cite[Conj.~9]{enge_ComputingClassPolynomials2014}. The goal of this
paper is to prove this fact. More precisely, we define
in~§\ref{sec:borchardt} a subset $\Fund'\subset\Half_2$
containing~$\Fund_2$, and prove the following result.

\begin{thm}
  \label{thm:main}
  For every $\tau\in\Fund'$, every $0\leq k\leq 3$ and every
  $n\geq 0$, the theta constants
  \begin{displaymath}
    \theta_j(2^n\gamma_k\tau) \quad \text{for } 0\leq j\leq 3
  \end{displaymath}
  are contained in a common open quarter plane.
\end{thm}

Dupont
\cite[Prop.~9.1]{dupont_MoyenneArithmeticogeometriqueSuites2006}
proved this result in the particular case of~$\gamma_0=I_4$.

As a consequence, we can invert genus 2 theta constants in
quasi-linear time by using only Borchardt sequences with good sign
choices.  On the practical side, this result reduces the effort needed
to invert genus~$2$ theta constants with controlled precision losses;
see for
instance~\cite[§7.4.2]{dupont_MoyenneArithmeticogeometriqueSuites2006}
for an analysis of precision losses when computing limits of Borchardt
sequences. On the theoretical side, we hope that our result can be a
first step towards removing other heuristic assumptions when computing
genus~$2$ theta constants (in particular, the
assumption~\cite[§10.2]{dupont_MoyenneArithmeticogeometriqueSuites2006}
that the function used in the Newton scheme is analytic with
invertible Jacobian), and obtaining algorithms with uniform complexity
in~$\tau\in \Fund_2$.

This document is organized as follows. In \cref{sec:borchardt}, we
introduce our notational conventions. In \cref{sec:cusp}, we use the
action of the symplectic group to bring the matrices
$2^n\gamma_k\tau\in\Half_2$ closer to the cusp at infinity: this is
critical to obtain accurate information from the series
expansion~\eqref{eq:theta}. We give estimates on genus~$2$ theta
constants in \cref{sec:bounds}, and we finish the proof of the main
theorem in \cref{sec:proof}.

\paragraph{Acknowledgement.} The author would like to thank Aurel Page
and the anonymous referees for their careful reading and helpful
suggestions to improve the exposition.

\section{Theta constants and Borchardt sequences}
\label{sec:borchardt}

We define a \emph{Borchardt sequence} to be a sequence of complex
numbers
\begin{displaymath}
  {(s_b^{(n)})}_{b\in (\Z/2\Z)^2,\, n\geq 0}
\end{displaymath}
with the following property: for every $n\geq 0$, there
exist~$t_b^{(n)}$ for $b\in(\Z/2\Z)^2$ such that~$t_b^{(n)}$ is a
square root of~$s_b^{(n)}$, and
\begin{displaymath}
  s_b^{(n+1)} = \frac14 \sum_{b_1+b_2=b} t_{b_1}^{(n)} t_{b_2}^{(n)}
  \qquad\text{for each } b\in(\Z/2\Z)^2.
\end{displaymath}
The duplication formula \cite[p.\,221]{mumford_TataLecturesTheta1983}
states that for every $\tau\in\Half_2$, the sequence
\begin{displaymath}
  B(\tau) = \bigl(\theta_{0,b}^2(2^n\tau)\bigr)_{b\in\{0,1\}^2, n\geq 0}
\end{displaymath}
is a Borchardt sequence; the choice of square roots at each step is
given by the theta constants~$\theta_{0,b}(2^n\tau)$
themselves. By the series expansion~\eqref{eq:theta}, we have
\begin{displaymath}
  \theta_{0,b}(2^n\tau) =
  \sum_{m\in \Z^2} \exp\bigl(- 2^n \pi m^t \im(\tau) m\bigr)
  \exp\left(i\pi\left(2^n m^t \re(\tau) m + m^t b\right)\right). 
\end{displaymath}
When~$n$ tends to infinity, all the terms except~$m=0$ converge
rapidly to zero, because~$\im(\tau)$ is positive definite. Therefore
the Borchardt sequence~$B(\tau)$ converges to~$(1,1,1,1)$.

We say that a set of complex numbers is \emph{in good position} when
it is included in an open quarter plane seen from the origin, i.e.~a set of the form
\begin{displaymath}
  \left\{r \exp({i(\alpha_0+\alpha)}) \ | \ r>0 \text{ and } 0<\alpha<\pi/2 \right\}
\end{displaymath}
for some~$\alpha_0\in \R$. The property of being in good position is
invariant by nonzero complex scaling. A Borchardt sequence is given by
\emph{good sign choices} if for every~$n\geq 0$, the complex
numbers~$t_b^{(n)}$ for $b\in(\Z/2\Z)^2$ are in good position.

Let us now detail the algorithm to recover~$\tau\in\Half_2$ from its
theta quotients. We first introduce the
matrices~$\gamma_k\in\Sp_4(\Z)$ alluded to in the introduction. Let
\begin{displaymath}
  S_1 = \mat{1}{0}{0}{0},\ S_2 = \mat{0}{0}{0}{1},\ S_3 = \mat{0}{1}{1}{0},
\end{displaymath}
and define the matrix~$\gamma_k\in\Sp_4(\Z)$ for $0\leq k\leq 3$ by
\begin{displaymath}
  \gamma_0 = I_4, \text{ and } \gamma_k 
  = \mat{-I_2}{-S_k}{S_k}{-I+S_k^2} \text{ for } 1\leq k\leq 3.
\end{displaymath}
For convenience, we also introduce a numbering of theta constants
\cite[§6.2]{dupont_MoyenneArithmeticogeometriqueSuites2006}:
\begin{displaymath}
  \theta_{(a_0,a_1),(b_0,b_1)} =: \theta_j \quad \text{where } j = b_0+2b_1+4a_0+8a_1\in\Zint{0,15}.
\end{displaymath}
Assuming that the choices of square roots in the
sequences~$B(\gamma_k\tau)$ can be determined, we can
compute~$\tau\in\Fund_2$ from its theta quotients as follows.

\begin{algo}[{\cite[§9.2.3]{dupont_MoyenneArithmeticogeometriqueSuites2006}}]~

  \noindent
  \textbf{Input:} The projective vector of squares of theta
  constants~$\theta_j^2(\tau)$ for~$j\in \Zint{0,15}$, for
  some~$\tau\in \Half_2$.\\
  \textbf{Output:} The matrix~$\tau$.
  \begin{enumerate}
  \item For each $0\leq k\leq 3$, compute the first term
    of the sequence~$B(\gamma_k\tau)/\theta_0^2(\gamma_k\tau)$ using the
    transformation formulas for theta constants under~$\Sp_4(\Z)$ (see
    Igusa~\cite[Thm.~2 p.\,175 and
    Cor.~p.\,176]{igusa_ThetaFunctions1972}, or
    \cref{cor:theta-transf});
  \item For each~$0\leq k\leq 3$, compute $1/\theta_0^2(\gamma_k\tau)$
    as the limit of the Borchardt sequence
    $B(\gamma_k\tau)/\theta_0^2(\gamma_k\tau)$;
  \item Use the input and the newly computed
    $\theta_0^2(\gamma_0\tau) = \theta_0^2(\tau)$ to compute all squares
    of theta constants at~$\tau$;
  \item Recover $\tau = \mat{z_1}{z_3}{z_3}{z_2}$ using the relations
    given
    in~\cite[§6.3.1]{dupont_MoyenneArithmeticogeometriqueSuites2006}:
    \begin{displaymath}
      \theta_0^2(\gamma_1\tau) = -i z_1 \theta_4^2(\tau),
      \quad \theta_0^2(\gamma_2\tau) = -i z_2 \theta_8^2(\tau),
      \quad \theta_0^2(\gamma_3\tau) = - \det(\tau) \theta_0^2(\tau).
    \end{displaymath}
  \end{enumerate}
\end{algo}
  
In the sequel, we use the following notational
conventions. For~$\tau\in\Half_2$, we write
\begin{displaymath}
  \tau = \mat{z_1(\tau)}{z_3(\tau)}{z_3(\tau)}{z_2(\tau)}
  \quad \text{and}\quad 
  \begin{cases}
    x_j(\tau) = \re z_j(\tau)\\  y_j(\tau) = \im z_j(\tau)
  \end{cases}
  \quad\text{ for } 1\leq j\leq 3.
\end{displaymath}
For $1\leq j\leq 3$, we also write
\begin{displaymath}
  q_j(\tau) = \exp(-\pi y_j(\tau)).
\end{displaymath}
We denote by~$\lambda_1(\tau)$ the smallest eigenvalue of~$\im(\tau)$, and define
\begin{displaymath}
  r(\tau) = \min\Bigl\{\lambda_1(\tau), \frac{y_1(\tau)}{2}, \frac{y_2(\tau)}{2}\Bigr\}.
\end{displaymath}
We often omit the argument~$\tau$ to ease notation. We define~$\Fund'$
to be the set of all~$\tau\in\Half_2$ such that the following
conditions are satisfied:
\begin{equation}
  \label{eq:F2}
  \begin{aligned}
    \abs{x_j(\tau)}&\leq \frac12 \quad \text{for each } 1\leq j\leq 3,\\
    2\abs{y_3(\tau)}&\leq y_1(\tau)\leq y_2(\tau),\\
    y_1(\tau)&\geq \frac{\sqrt{3}}{2},\\
    \abs{z_j(\tau)}&\geq 1 \quad \text{for } j\in\{1,2\}.
  \end{aligned}
\end{equation}
The domain~$\Fund'$ contains the classical fundamental domain~$\Fund_2$
for the action of~$\Sp_4(\Z)$ on~$\Half_2$ \cite[Prop.~3
p.~33]{klingen_IntroductoryLecturesSiegel1990}. Assumptions similar
to~\eqref{eq:F2} are usual when giving analytic estimates on theta
constants: for instance, the domain~$\mathcal{B}$
in~\cite{streng_ComputingIgusaClass2014} is defined by the first three
inequalities of~\eqref{eq:F2}.

Finally, for each~$\tau\in \Half_2$, we write
\begin{equation}
  \label{eq:xi}
  \begin{aligned}
    \xi_{4,6}(\tau) &= 2\exp \Bigl(i \pi \frac{z_1(\tau)}{4} \Bigr),\\
    \xi_{8,9}(\tau) &= 2\exp \Bigl(i \pi \frac{z_2(\tau)}{4} \Bigr),\\
    \xi_0(\tau) &= 1 + 2\exp(i\pi z_1(\tau)) + 2\exp(i\pi z_2(\tau)), \\
    \xi_{0,2}(\tau) &= 1 + 2\exp(i\pi z_1(\tau)),\\
    \xi_{0,1}(\tau) &= 1 + 2\exp(i\pi z_2(\tau)), \quad\text{and}\\
    \xi_{12}(\tau) &= \exp\left(i\pi\frac{z_1(\tau)+z_2(\tau)}{4}\right)
    \left(
      \exp\left(i\pi \frac{z_3(\tau)}{2}\right)
      +\exp\left(-i\pi \frac{z_3(\tau)}{2}\right)\right).
  \end{aligned}
\end{equation}
These complex numbers correspond to the first term(s) of the series
defining theta constants at~$\tau$. For instance,~$\xi_{4,6}(\tau)$
approximates both~$\theta_4(\tau)$ and~$\theta_6(\tau)$. We will
recall the definitions~\eqref{eq:xi} before using them in the
computations of~§\ref{sec:bounds}.

\section{Other expressions for theta constants at \texorpdfstring{$2^n\gamma_k\tau$}{}}
\label{sec:cusp}

For every $n\geq 0$, we define
\begin{align*}
  \eta_1^{(n)} &=
  \left(
    \begin{matrix}
      0&0&-1&0\\0&1&0&0\\1&0&2^n&0\\0&0&0&1
    \end{matrix}
  \right),\quad\ \
  \eta_2^{(n)} =
  \left(
    \begin{matrix}
      1&0&0&0\\0&0&0&-1\\0&0&1&0\\0&1&0&2^n
    \end{matrix}
  \right),\\
  \eta_3^{(n)} &=
  \left(
    \begin{matrix}
      0&0&0&-1\\0&0&-1&0\\0&1&2^n&0\\1&0&0&2^n
    \end{matrix}
  \right),\quad \text{and}\quad
  \eta_4^{(n)} =
  \left(
    \begin{matrix}
      0&0&1&0\\0&1&0&0\\-1&0&0&0\\0&0&0&1
    \end{matrix}
  \right) \eta_3^{(n)}.
\end{align*}

\begin{lem} Let $n\geq 0$.
  \begin{enumerate}
  \item For every $1\leq k\leq 4$, the matrix~$\eta_k^{(n)}$ belongs
    to~$\Sp_4(\Z)$.
  \item For every $\tau = \mat{z_1}{z_3}{z_3}{z_2}\in\Half_2$, we have
    \begin{equation}
      \label{eq:taujn}
    \begin{aligned}
      \tau_1^{(n)} :=  \eta_1^{(n)}(2^n\gamma_1\tau) &=  \mat{2^{-n}z_1}{z_3}{z_3}{2^nz_2},\\
      \tau_2^{(n)} :=  \eta_2^{(n)}(2^n\gamma_2\tau) &=  \mat{2^nz_1}{z_3}{z_3}{2^{-n}z_2},\\
      \tau_3^{(n)} :=  \eta_3^{(n)}(2^n\gamma_3\tau) &=  2^{-n}\tau, \quad\text{and}\\
      \tau_4^{(n)} := \eta_4^{(n)}(2^n\gamma_3\tau) &=
      \mat{-2^n/z_1}{-z_3/z_1}{-z_3/z_1}{2^{-n}(z_2-z_3^2/z_1)}.
    \end{aligned}
    \end{equation}
  \end{enumerate}
\end{lem}

\begin{proof}
  \begin{enumerate}
  \item The lines of each~$\eta_k^{(n)}$ define a
    symplectic basis of~$\Z^4$.
  \item The action of~$\Sp_4(\Z)$ on~$\Half_2$ extends to an action of
    the larger group
    \begin{displaymath}
      \GSp_4(\Q) = \left\{\gamma\in \GL_4(\Q) \ |\  \exists \mu\in \Q^\times,
        \gamma^t \mat{0}{I_2}{-I_2}{0} \gamma = \mu \mat{0}{I_2}{-I_2}{0}
      \right\}.
    \end{displaymath}
    The matrix~$2^n\gamma_k\tau$ is the image of~$\tau$ under
    \begin{displaymath}
      \mat{-2^n I_2}{-2^n S_k}{S_k}{-I + S_i^2} \in \GSp_4(\Q).
    \end{displaymath}
    When we multiply this matrix by~$\eta_k^{(n)}$ on the left, we obtain
    \begin{align*}
      \Diag(-1,-2^n,-2^n,-1) &\qquad\text{for } k=1,\\
      \Diag(-2^n,-1,-1,-2^n) &\qquad\text{for } k=2, \text{ and}\\
      \Diag(-1,-1,-2^n,-2^n) &\qquad\text{for } k=3. \qedhere
    \end{align*}
  \end{enumerate}
\end{proof}

We recall the transformation formulas for theta constants in
genus~$2$. For a square matrix~$m$, we denote by~$m_0$ the column
vector containing the diagonal of~$m$.

\begin{prop}[{\cite[Thm.~2 p.\,175 and
    Cor.~p.\,176]{igusa_ThetaFunctions1972}}]
  \label{prop:transf}
  Let $a, b\in \{0,1\}^2$, and let
  \begin{displaymath}
    \gamma = \mat{A}{B}{C}{D}\in\Sp_4(\Z).
  \end{displaymath}
  Define
  \begin{displaymath}
    \vect{\alpha}{\beta} = \gamma^t \vect{a - (C D^t)_0}{b - (A B^t)_0}.
  \end{displaymath}
  Then, for every $\tau\in\Half_2$, we have
  \begin{displaymath}
    \theta_{a,b}(\gamma\tau) = \kappa(\gamma)\, \zeta_8^{\eps(\gamma,a,b)}
    \det(C\tau+D)^{1/2}\,\theta_{a',b'}(\tau)
  \end{displaymath}
  where 
  \begin{align*}
    \zeta_8&=e^{i\pi/4},\qquad \vect{a'}{b'} = \vect{\alpha}{\beta} \mod 2,\\
    \eps(\gamma,a,b) &= 2 (B\alpha)^t (C\beta) - (D\alpha)^t(B\alpha)
                       - (C\beta)^t(A\beta) + 2 ((A B^t)_0)^t (D\alpha-C\beta),
  \end{align*}
  and~$\kappa(\gamma)$ is an eighth root of unity depending only
  on~$\gamma$, with a sign ambiguity coming from the choice of a
  holomorphic square root of~$\det(C\tau+D)$.
\end{prop}

\begin{cor} \label{cor:theta-transf}
  For every $\tau\in \Half_2$, we have the following equalities of projective tuples:
  \begin{align*}
    (\theta_j(2^n\gamma_1\tau))_{0\leq j\leq 3}
    &=
      \begin{cases}
        (\theta_4(\tau) :
        \theta_0(\tau) :
        \theta_6(\tau) :
        \theta_2(\tau)) &\text{if } n=0,\\
        (\theta_0(\tau_1^{(n)}) :
        \theta_4(\tau_1^{(n)}) :
        \theta_2(\tau_1^{(n)}) :
        \theta_6(\tau_1^{(n)})) &\text{if } n\geq1,
      \end{cases}
    \\
    (\theta_j(2^n\gamma_2\tau))_{0\leq j\leq 3}
    &=
      \begin{cases}
        (\theta_8(\tau) :
        \theta_9(\tau) :
        \theta_0(\tau) :
        \theta_1(\tau)) &\text{if } n=0,\\
        (\theta_0(\tau_2^{(n)}) :
        \theta_1(\tau_2^{(n)}) :
        \theta_8(\tau_2^{(n)}) :
        \theta_9(\tau_2^{(n)})) &\text{if } n\geq1,
      \end{cases}
    \\
    (\theta_j(2^n\gamma_3\tau))_{0\leq j\leq 3}
    &=
    (\theta_0(\tau_3^{(n)}) :
    \theta_8(\tau_3^{(n)}) :
    \theta_4(\tau_3^{(n)}) :
    \theta_{12}(\tau_3^{(n)})) \quad\text{ for every } n\geq 0,
    \\
    (\theta_j(2^n\gamma_3\tau))_{0\leq j\leq 3}
    &=
    (\theta_0(\tau_4^{(n)}) :
    \theta_8(\tau_4^{(n)}) :
    \theta_1(\tau_4^{(n)}) :
    \theta_{9}(\tau_4^{(n)})) \quad\ \text{ for every } n\geq 0,
  \end{align*}
  where the $\tau_j^{(n)}$ are defined as in~\eqref{eq:taujn}.
\end{cor}

\begin{proof}
  Apply \cref{prop:transf} to the matrices $\eta_i^{(n)}$.
\end{proof}

When $\tau\in\Fund'$, the real and imaginary parts of $\tau_k^{(n)}$
for $1\leq k\leq 3$ are easy to study: for instance, from the second
inequality in~\eqref{eq:F2} we always have
\begin{displaymath}
  y_3(\tau_k^{(n)})^2\leq \frac14 y_1(\tau_k^{(n)}) y_2(\tau_k^{(n)}).
\end{displaymath}
Such estimates are less obvious for the matrices $\tau_4^{(n)}$.

\begin{lem}
  \label{lem:tau4}
  Let $\tau\in\Fund'$. Then, for every $n\geq 0$, we have
  \begin{align*}
    \babs{y_3(\tau_4^{(n)})} &\leq \frac{3}{2^{n+2}} y_1(\tau_4^{(n)}), \\
    y_3(\tau_4^{(n)})^2 &\leq \frac37 y_1(\tau_4^{(n)}) y_2(\tau_4^{(n)}), \quad\text{and}\\
    \babs{x_2(\tau_4^{(n)})} &\leq \frac{9}{2^{n+3}}.
  \end{align*}
\end{lem}

\begin{proof}
  Write~$z_1$ for $z_1(\tau)$, etc. We have
  \begin{displaymath}
    y_3(\tau_4^{(n)}) = \im(-z_3/z_1) = \frac{1}{\abs{z_1}^2}(x_3y_1-y_3x_1),
  \end{displaymath}
  so
  \begin{displaymath}
    \babs{y_3(\tau_4^{(n)})} \leq \frac{3y_1}{4\abs{z_1}^2} = \frac{3}{2^{n+2}} y_1(\tau_4^{(n)}),
  \end{displaymath}
  since~$y_1(\tau_4^{(n)}) = 2^{n} y_1/\abs{z_1}^2$
  by~\eqref{eq:taujn}. For the second inequality, we have
  \begin{displaymath}
    \im(\tau_4^{(n)}) = {\mat{2^{-n}z_1}{-2^{-n}z_3}{0}{1}}^{-t} (2^{-n}\im \tau)
    {\mat{2^{-n}\conj{z}_1}{-2^{-n}\conj{z}_3}{0}{1}}^{-1}
  \end{displaymath}
  so
  \begin{displaymath}
    \det\im(\tau_4^{(n)}) = \frac{1}{\abs{z_1}^2}\det\im\tau.
  \end{displaymath}
  Moreover $\det\im\tau \geq \frac34y_1^2$, so
  \begin{align*}
    \frac{y_3(\tau_4^{(n)})^2}{y_1(\tau_4^{(n)})y_2(\tau_4^{(n)})} &\leq \frac{y_3(\tau_4^{(n)})^2}{y_3(\tau_4^{(n)})^2 + \frac{3y_1^2}{4\abs{z_1}^2}} \leq \frac{1}{1 + \frac43 \abs{z_1}^2} \leq \frac37.
  \end{align*}
  For the last inequality, we compute
  \begin{displaymath}
    2^n x_2(\tau_4^{(n)}) = x_2 - \frac{1}{\abs{z_1}^2}((x_3^2 - y_3^2)x_1 + 2x_3 y_3 y_1)
  \end{displaymath}
  and
  \begin{displaymath}
    \bbabs{\frac{1}{\abs{z_1}^2}(x_3^2 - y_3^2)x_1} \leq \frac12 \max\{x_3^2, \frac{y_3^2}{\abs{z_1}^2}\} \leq \frac18,
  \end{displaymath}
  so
  \begin{displaymath}
    \babs{2^n x_2(\tau_4^{(n)})} \leq \frac12 + \frac18 + \frac12 = \frac98. \qedhere
  \end{displaymath}
\end{proof}

\section{Bounds on theta constants}
\label{sec:bounds}

Typically, when $\tau\in\Half_2$ is close enough to the cusp at
infinity (more precisely when~$\im z_1(\tau),\im z_2(\tau)$,
and~$\det \im(\tau)$ are large), useful information on theta constants
at~$\tau$ can be obtained from the series expansion
\eqref{eq:theta}. Our computations are similar in spirit to those
found in~\cite[pp.\,116--117]{klingen_IntroductoryLecturesSiegel1990},
\cite[§6.2]{dupont_MoyenneArithmeticogeometriqueSuites2006},
\cite[§5.1]{habegger_BadReductionGenus2017}.  All our estimates are
based on the following key lemma.

\begin{lem}
  \label{lem:tail}
  Let $f\from \N\to \R$ be a strictly increasing function, and assume
  that $f(k+2)-f(k+1) \geq f(k+1) - f(k)$ for every $k\geq 0$. Let
  $0<q<1$. Then
  \begin{displaymath}
    \sum_{k=0}^\infty q^{f(k)} \leq \frac{q^{f(0)}}{1 - q^{f(1)-f(0)}}.
  \end{displaymath}
\end{lem}

\begin{proof}
  Use that $f(k) \geq f(0) + k(f(1)-f(0))$ for all $k$.
\end{proof}

\begin{lem}
  \label{lem:bound46}
  Let $k\geq 1$, and let $\tau\in\Half_2$ such that
  \begin{displaymath}
    y_3(\tau)^2 \leq \frac14 y_1(\tau) y_2(\tau) \quad\text{and}\quad
    k \abs{y_3(\tau)}\leq y_2(\tau).
  \end{displaymath}
  Define
  \begin{align*}
    \xi_{4,6}(\tau) &= 2\exp \Bigl(i \pi \frac{z_1(\tau)}{4} \Bigr)
  \end{align*}
  and
  \begin{align*}
    \rho_{4,6}^{(k)}(q_1,q_2)
        &= \frac{q_1^2}{1-q_1^4} + \frac{q_2^{1-\frac1k}}{1-q_2^{3-\frac1k}}
          + \frac{q_2^{1+\frac1k}}{1-q_2^{3+\frac1k}} \\
    &\qquad 
          + \frac{ q_1^{7/8} q_2^{1/2}}{(1-q_2^{3/2})(1-q_1^2)}
          + \frac{q_1^{25/8} q_2^{3/2}}{(1-q_2^{9/2})(1- q_1^6)}.
  \end{align*}
  Then for $j\in\{4,6\}$, we have
  \begin{displaymath}
    \abs{\frac{\theta_j(\tau)}{\xi_{4,6}(\tau)} -1} \leq \rho_{4,6}^{(k)}\bigl(q_1(\tau),q_2(\tau)\bigr).
  \end{displaymath}
\end{lem}

\begin{proof}
  Write $u = \svec{1/2}{0}$. Using the definition, we obtain
  \begin{align*}
    \abs{\frac{\theta_j(\tau)}{\xi_{4,6}(\tau)} -1}
    &\leq \frac12\,
      q_1^{-1/4}
      \sum_{\substack{m\in\Z^2\\m\neq \svec{0}{0},\svec{-1}{0}}} \exp\bigl(-\pi (m+u)^t \im(\tau) (m+u)\bigr).
  \end{align*}
  We split this sum in two parts, according to whether the second
  coordinate of~$m$ is zero or not. The first part gives
  \begin{displaymath}
    q_1^{-1/4} \sum_{m\in \N + \frac32} q_1^{m^2}
    \leq q_1^{-1/4} \frac{q_1^{9/4}}{1 - q_1^{4}} = \frac{q_1^2}{1-q_1^4}
  \end{displaymath}
  by \cref{lem:tail}.  The second part is
  \begin{displaymath}
    q_1^{-1/4}
    \sum_{m_1\in \N+\frac12} \sum_{m_2\geq 1}
    q_1^{m_1^2} q_2^{m_2^2}
    \cdot 2\cosh(2\pi y_3 m_1 m_2).
  \end{displaymath}
  We use the fact that for every $(m_1,m_2)\in \R_+^2$,
  \begin{align*}
    \babs{2y_3m_1m_2}   &\leq \frac{y_1}{2} m_1^2 + \frac{y_2}{2} m_2^2.
  \end{align*}
  When $m_1 = 1/2$, we use the following bound instead:
  \begin{displaymath}
    \babs{2y_3 m_1m_2} = \abs{y_3m_2} \leq \frac{y_2 m_2}{k}.
  \end{displaymath}
  Therefore the total contribution of the second part is bounded by
  \begin{align*}
    &q_1^{-1/4} \sum_{m_2\geq 1} q_1^{1/4} q_2^{m_2^2}
      \cdot 2\cosh\left(\pi \frac{y_2}{k} m_2\right) \\
    &\quad + q_1^{-1/4} \sum_{m_1\in\N+\frac32} \sum_{m_2\geq 1}
      q_1^{m_1^2}q_2^{m_2^2}
      \cdot 2\cosh\left(\pi \left(\frac{y_1}{2} m_1^2+ \frac{y_2}{2}m_2^2\right)\right) \\
    &\leq \frac{q_2^{1 - \frac1k}}{1-q_2^{3-\frac1k}} + \frac{q_2^{1+\frac1k}}{1-q_2^{3+\frac1k}}
      + \frac{ q_1^{7/8} q_2^{1/2}}{(1-q_2^{3/2})(1-q_1^2)} + \frac{q_1^{25/8} q_2^{3/2}}{(1-q_2^{9/2})(1- q_1^6)}
  \end{align*}
  by other applications of \cref{lem:tail}.
\end{proof}

\begin{lem}
  \label{lem:bound89}
  Let $k\geq 1$, and let $\tau\in\Half_2$ such that
  \begin{displaymath}
    y_3(\tau)^2 \leq \frac14 y_1(\tau) y_2(\tau) \quad\text{and}\quad
    k \abs{y_3(\tau)}\leq y_1(\tau).
  \end{displaymath}
  Define
  \begin{align*}
    \xi_{8,9}(\tau) &= 2\exp \Bigl(i \pi \frac{z_2(\tau)}{4} \Bigr),
  \end{align*}
  and
  \begin{align*}
    \rho_{8,9}^{(k)}(q_1,q_2)
        &= \frac{q_2^2}{1-q_2^4} + \frac{q_1^{1-\frac1k}}{1-q_1^{3-\frac1k}}
          + \frac{q_1^{1+\frac1k}}{1-q_1^{3+\frac1k}} \\
    &\qquad 
          + \frac{ q_2^{7/8} q_1^{1/2}}{(1-q_1^{3/2})(1-q_2^2)}
          + \frac{q_2^{25/8} q_1^{3/2}}{(1-q_1^{9/2})(1- q_2^6)}.
  \end{align*}
  Then for $j\in\{8,9\}$, we have
  \begin{displaymath}
    \abs{\frac{\theta_j(\tau)}{\xi_{8,9}(\tau)} -1} \leq \rho_{8,9}^{(k)}\bigl(q_1(\tau),q_2(\tau)\bigr).
  \end{displaymath}
\end{lem}

\begin{proof}
  We proceed in a similar fashion as in the proof of~\cref{lem:bound46} by
  switching the roles of $q_1$ and $q_2$.
\end{proof}

\begin{lem}
  \label{lem:bound012}
  Let $\tau\in\Half_2$ such that
  \begin{displaymath}
    y_3(\tau)^2 \leq \frac14 y_1(\tau) y_2(\tau).
  \end{displaymath}
  Define
  \begin{align*}
    \xi_0(\tau) &= 1 + 2\exp(i\pi z_1(\tau)) + 2\exp(i\pi z_2(\tau)), \\
    \xi_{0,2}(\tau) &= 1 + 2\exp(i\pi z_1(\tau)),\\
    \xi_{0,1}(\tau) &= 1 + 2\exp(i\pi z_2(\tau)),
  \end{align*}
  and
  \begin{align*}
    \rho_0(q_1,q_2) &= \frac{2q_1^4}{1-q_1^5} + \frac{2q_2^4}{1-q_2^5}
      + \frac{2 q_1^{1/2} q_2^{1/2}}{(1-q_1^{3/2})(1-q_2^{3/2})}
      + \frac{2 q_1^{3/2} q_2^{3/2}}{(1-q_1^{9/2})(1-q_2^{9/2})}.
  \end{align*}
  Then we have
  \begin{align*}
    \abs{\theta_0(\tau) - \xi_0(\tau)} &\leq \rho_0(q_1(\tau),q_2(\tau)), \\
    \abs{\theta_j(\tau) - \xi_{0,2}(\tau)} &\leq \rho_0(q_1(\tau),q_2(\tau)) + 2q_2(\tau)
                                             &\text{for } j\in\{0,2\},\\    
    \abs{\theta_j(\tau) - \xi_{0,1}(\tau)} &\leq \rho_0(q_1(\tau),q_2(\tau)) + 2q_1(\tau)
                                             &\text{for } j\in\{0,1\},&\text{ and}\\
    \abs{\theta_j(\tau) - 1}&\leq \rho_0(q_1(\tau),q_2(\tau)) + 2q_1(\tau) + 2q_2(\tau)
                              &\text{for } 0\leq j\leq 3.
  \end{align*}
\end{lem}

\begin{proof}
  We proceed again in a similar fashion as in the proof
  of~\cref{lem:bound46}. The terms of~$\rho_0(q_1,q_2)$ are obtained
  by considering the following subsets of indices~$m\in \Z^2$:
  \begin{displaymath}
    \{\svec{m_1}{0} \ |\ \abs{m_1}\geq 2\}, \quad \{\svec{0}{m_2}\ |\ \abs{m_2}\geq 2\},
  \end{displaymath}
  and
  \begin{displaymath}
    \quad\{\svec{m_1}{m_2}\ |\ \abs{m_1}\geq 1,\abs{m_2}\geq 1\}. \qedhere
  \end{displaymath}
\end{proof}

\begin{lem}
  \label{lem:bound12}
  Let $\tau\in\Half_2$ such that
  \begin{displaymath}
    \abs{x_3(\tau)}\leq \frac12 \quad \text{and}\quad
    2 \abs{y_3(\tau)} \leq \min\{y_1(\tau),y_2(\tau)\}.
  \end{displaymath}
  Write
  \begin{align*}
    \xi_{12}(\tau)
    &= \exp\left(i\pi\frac{z_1(\tau)+z_2(\tau)}{4}\right)
      \left(
      \exp\left(i\pi \frac{z_3(\tau)}{2}\right)
      +\exp\left(-i\pi \frac{z_3(\tau)}{2}\right)\right),
  \end{align*}
  and
  \begin{align*}
    \rho_{12}(q_1,q_2) &= \frac{q_1^{3/2}}{1-q_1^{7/2}} + \frac{q_1^{5/2}}{1 - q_1^{9/2}}
                         + \frac{q_2^{3/2}}{1-q_2^{7/2}} + \frac{q_2^{5/2}}{1-q_2^{9/2}} \\
        &\qquad + \frac{q_1^{7/8}q_2^{7/8}}{(1-q_1^2)(1-q_2^2)}
          + \frac{q_1^{25/8}q_2^{25/8}}{(1-q_1^6)(1-q_2^6)}.
  \end{align*}
  Then we have
  \begin{displaymath}
    \abs{\frac{\theta_{12}(\tau)}{\xi_{12}(\tau)}-1} \leq \rho_{12}\bigr(q_1(\tau),q_2(\tau)\bigr).
  \end{displaymath}
\end{lem}

\begin{proof}
  By~\eqref{eq:theta}, we have
  \begin{align*}
    \theta_{12}(\tau)
    = 2\sum_{m_1\in \N+\frac12} \sum_{m_2\in\N+\frac12}
    &\exp
       \bigl(i\pi(m_1^2z_1 + m_2^2z_2)\bigr) \\
    &\cdot
      \bigl(\exp(2\pi i m_1m_2 z_3)+ \exp(-2\pi i m_1m_2z_3)\bigr).
  \end{align*}
  We leave the term corresponding to $(m_1,m_2) = (\frac12,\frac12)$ aside, and write
  \begin{align*}
    &\abs{\frac{\theta_{12}(\tau)}{2\exp(i\pi(z_1+z_2)/4)} -
      (\exp(i\pi z_3/2)+\exp(-i\pi z_3/2))} \\
    &\leq \sum_{\substack{(m_1,m_2)\in(\N+\frac12)^2\\(m_1,m_2)\neq(\frac12,\frac12)}}
    q_1^{m_1^2-\frac14} q_2^{m_2^2-\frac14}\cdot 2\cosh(2\pi m_1m_2 y_3).
  \end{align*}
  Since $\abs{x_3}\leq \frac12$, the absolute value of the argument of
  $\exp(i\pi z_3/2)$ is at most~$\pi/4$. Therefore,
  \begin{displaymath}
    \babs{\exp(i\pi z_3/2) + \exp(-i\pi z_3/2)} \geq \exp(\pi\abs{y_3}/2).
  \end{displaymath}
   We obtain
  \begin{align*}
    \abs{\frac{\theta_{12}(\tau)}{\xi_{12}(\tau)} - 1}
    &\leq \sum_{\substack{(m_1,m_2)\in(\N+\frac12)^2\\(m_1,m_2)\neq(\frac12,\frac12)}}
    q_1^{m_1^2-\frac14} q_2^{m_2^2-\frac14}
    \cdot 2\cosh\bigl(2\pi (m_1m_2-\tfrac14) y_3\bigr).
  \end{align*}
  We separate the terms corresponding to $m_2=\frac12$. Since
  $2\abs{y_3}\leq y_1$, their contribution is bounded by
  \begin{align*}
    \sum_{m_1\in \N+\frac32} \Bigl( q_1^{m_1^2 - \frac 12 m_1} + q_1^{m_1^2 + \frac12 m_1 - \frac12} \Bigr)
    &\leq \frac{q_1^{3/2}}{1-q_1^{7/2}} + \frac{q_1^{5/2}}{1 - q_1^{9/2}}.
  \end{align*}
  Similarly, the contribution from the terms with $m_1=1/2$ is bounded by
  \begin{displaymath}
    \frac{q_2^{3/2}}{1-q_2^{7/2}} + \frac{q_2^{5/2}}{1-q_2^{9/2}}.
  \end{displaymath}
  For the remaining terms, we use the majoration
  \begin{displaymath}
    \babs{2\pi (m_1m_2-\tfrac14) y_3} \leq \babs{2\pi m_1m_2y_3}
    \leq \pi\abs{y_3}(m_1^2+m_2^2) \leq \frac\pi2 (m_1^2 y_1 + m_2^2 y_2).
  \end{displaymath}
  Thus, the rest of the sum is bounded by
  \begin{align*}
    &\sum_{m_1,m_2\in\N+\frac32} 
      q_1^{m_1^2-\frac14} q_2^{m_2^2-\frac14}
      \cdot 2\cosh\left(\frac{\pi}{2}(m_1^2y_1 + m_2^2 y_2)\right) \\
    &\leq \sum_{m_1,m_2\in\N+\frac32} \Bigl(
      q_1^{\frac12 m_1^2 - \frac14} q_2^{\frac12 m_2^2 -\frac14}
      + 
      q_1^{\frac32 m_1^2 - \frac14} q_2^{\frac32 m_2^2 - \frac14}
      \Bigr) \\
    &\leq \frac{q_1^{7/8}q_2^{7/8}}{(1-q_1^2)(1-q_2^2)} + \frac{q_1^{25/8}q_2^{25/8}}{(1-q_1^6)(1-q_2^6)}.
      \qedhere
  \end{align*}
\end{proof}

We give another version of these estimates that we will use for~$\tau_4^{(n)}$.

\begin{lem}
  \label{lem:bound0prime}
  Let $k\geq 2$, and let $\tau\in\Half_2$ such that
  \begin{displaymath}
    y_3(\tau)^2\leq \frac{3}{7}\, y_1(\tau) y_2(\tau) \quad\text{and}\quad
    k\abs{y_3(\tau)} \leq y_1(\tau).
  \end{displaymath}
  Let $\alpha = \sqrt{3/7}$. Define
  \begin{align*}
    \rho_{0,1}'^{(k)}(q_1,q_2)
    &= 
    \frac{2q_2^4}{1-q_2^5} + \frac{2q_1}{1-q_1^3} + \frac{2 q_1^{1-\frac2k}q_2}{1-q_1^{3-\frac2k}}
    + \frac{2q_1^{1+\frac2k}q_2}{1-q_1^{3+\frac2k}}\\
    &\qquad + \frac{2q_1^{1-\alpha} q_2^{4(1-\alpha)}}{(1-q_1^{3(1-\alpha)})(1-q_2^{5(1-\alpha)})}
      + \frac{2q_1^{1+\alpha} q_2^{4(1+\alpha)}}{(1-q_1^{3(1+\alpha)})(1-q_2^{5(1+\alpha)})}
  \end{align*}
  and
  \begin{align*}
    \rho_{8,9}'^{(k)}(q_1,q_2)
    &=
        \frac{q_2^2}{1-q_2^4} + \frac{q_1^{1-\frac1k}}{1-q_1^{3-\frac1k}}
        + \frac{q_1^{1+\frac1k}}{1-q_1^{3+\frac1k}} \\
    &\quad + \frac{q_2^{2-\frac94\alpha}q_1^{1-\alpha}}{(1-q_2^{4(1-\alpha)})(1-q_1^{3(1-\alpha)})}
      + \frac{q_2^{2+\frac94\alpha}q_1^{1+\alpha}}{(1-q_2^{4(1+\alpha)})(1-q_1^{3(1+\alpha)})}.
  \end{align*}
  Then we have
  \begin{align*}
    \abs{\theta_j(\tau)-\xi_{0,1}(\tau)} \leq \rho_{0,1}'(\tau) \quad \text{for } j\in\{0,1\}
  \end{align*}
  and
  \begin{align*}
    \abs{\frac{\theta_j(\tau)}{\xi_{8,9}(\tau)}-1} \leq \rho_{8,9}'(\tau) \quad \text{for } j\in\{8,9\}.
  \end{align*}
\end{lem}

\begin{proof}
  We bound the cross-product terms by
  \begin{align*}
    \babs{2y_3m_1m_2}&\leq \alpha y_1 m_1^2 + \alpha y_2m_2^2,\\
    \babs{2y_3m_1m_2}&\leq \frac 1k y_1 m_1 &\text{if } m_2=\frac12, &\quad\text{and}\\
    \babs{2y_3m_1m_2}&\leq \frac 2k y_1 m_1  &\text{if } m_2=1.
  \end{align*}
  For $j\in\{0,1\}$, we separate the terms with $\abs{m_2}\leq 1$ or
  $m_1=0$, and obtain
  \begin{align*}
    \abs{\theta_j(\tau)-\xi_{0,1}(\tau)}
    &\leq 2\sum_{m_2\geq 2} q_2^{m_2^2} + 2\sum_{m_1\geq 1} q_1^{m_1^2}
          +
          2\sum_{m_2\geq 1} q_2(q_1^{m_1^2 - \frac2k m_1} + q_1^{m_1^2+\frac2k m_1}) \\
    &\qquad + 2\sum_{m_1\geq 1} \sum_{m_2\geq 2} q_1^{m_1^2} q_2^{m_2^2}
      \cdot 2\cosh\left(\alpha(y_1m_1^2+ y_2m_2^2)\right) \\
    &\leq \frac{2q_2^4}{1-q_2^5} + \frac{2q_1}{1-q_1^3} + \frac{2 q_1^{1 - \frac2k}q_2}{1-q_1^{3-\frac2k}} + \frac{2q_1^{1+\frac2k}q_2}{1-q_1^{3+\frac2k}}\\
    &\qquad + \frac{2q_1^{1-\alpha} q_2^{4(1-\alpha)}}{(1-q_1^{3(1-\alpha)})(1-q_2^{5(1-\alpha)})}
      + \frac{2q_1^{1+\alpha} q_2^{4(1+\alpha)}}{(1-q_1^{3(1+\alpha)})(1-q_2^{5(1+\alpha)})}.
  \end{align*}
  For $j\in\{8,9\}$, we separate the terms with $\abs{m_2}=\frac12$ or $m_1=0$. We obtain
  \begin{align*}
    \abs{\frac{\theta_j(\tau)}{\xi_{8,9}(\tau)} - 1}
    &\leq q_2^{-1/4} \sum_{m_2\in \N+\frac32} q_2^{m_2^2} + \sum_{m_1\geq 1} \left(q_1^{m_1^2 - \frac1k m_1} + q_1^{m_1^2 + \frac1k m_1}\right)\\
    &\quad + q_2^{-1/4} \sum_{m_2\in\N+\frac32} \sum_{m_1\geq 1} q_2^{m_2^2} q_1^{m_2^2}
      \cdot 2\cosh\left(\alpha(y_1m_1^2+ y_2m_2^2)\right)
    \\
    &\leq
      \frac{q_2^2}{1-q_2^4} + \frac{q_1^{1-\frac1k}}{1-q_1^{3-\frac1k}}
      + \frac{q_1^{1+\frac1k}}{1-q_1^{3+\frac1k}} \\
    &\quad + \frac{q_2^{2-\frac94\alpha}q_1^{1-\alpha}}{(1-q_2^{4(1-\alpha)}(1-q_1^{3(1-\alpha)})}
      + \frac{q_2^{2+\frac94\alpha}q_1^{1+\alpha}}{(1-q_2^{4(1+\alpha)})(1-q_1^{3(1+\alpha)})}.
      \qedhere
  \end{align*}
\end{proof}

Finally, when~$n$ is large, we will show that the theta
constants~$\theta_j(2^n\gamma_k\tau)$ for $0\leq j\leq 3$ are in good
position using the following lemma. Recall the definition of~$r(\tau)$
and~$\lambda_1(\tau)$ from~§\ref{sec:borchardt}.

\begin{lem} \label{lem:goodpos}
  Let $\tau\in\Half_2$.
  \begin{enumerate}
  \item 
    If $r(\tau)\geq 0.4$, then the $\theta_j(\tau)$ for
    $0\leq j\leq 3$ are in good position.
  \item If $\lambda_1(\tau)\geq 0.6$, then the $\theta_j(\tau)$ for
    $0\leq j\leq 3$ are in good position.
  \end{enumerate}
\end{lem}

\begin{proof}
  \begin{enumerate}
  \item 
    Write
    \begin{displaymath}
      q = \exp(-\pi r(\tau)).
    \end{displaymath}
    For $0\leq j\leq 3$, we have
    \begin{equation}
      \label{eq:theta0123}
      \begin{aligned}
        \abs{\theta_j(\tau)-1}
        &\leq 4q^2 + \sum_{n\in\Z^2,\, \norm{n}^2\geq 2} \exp(-\pi \lambda_1(\tau) \norm{n}^2) \\
        &\leq 8q^2 + 4q^4 + 8q^5 + 4q^8 + 4\dfrac{1+q}{(1-q)^2}q^9.
      \end{aligned}
    \end{equation}
    In this inequality, the first term~$4q^2$ comes from the four
    vectors~$n\in\Z^2$ with~$\norm{n}=1$. Then we separate the
    terms~$n = \svec{n_1}{n_2}$ such that $\abs{n_1}\geq 3$
    and~$\abs{n_2}\geq 3$; this accounts for the
    term~$4q^9(1+q)/(1-q)^2$, as in the proof
    of~\cite[Prop.~6.1]{dupont_MoyenneArithmeticogeometriqueSuites2006}. We
    leave the remaining terms as they are.
    
    If $q \leq 0.287$, then the quantity on the right hand side
    of~\eqref{eq:theta0123} is less than~$\sqrt{2}/2$, and the
    $\theta_j(\tau)$ are contained in a disk which is itself contained
    in a quarter plane. We have $q\leq 0.287$ when $r(\tau)\geq 0.4$.
  \item Write
    \begin{displaymath}
      q = \exp(-\pi \lambda_1(\tau)).
    \end{displaymath}
    Then for $0\leq j\leq 3$, we have
    \begin{align*}
      \abs{\theta_j(\tau)-1}
      &\leq 4q + 4q^2 + 4q^4 + 8q^5 + 4q^8 + 4\dfrac{1+q}{(1-q)^2}q^9.
    \end{align*}
    This quantity is less than $\sqrt{2}/2$ when
    $\lambda_1(\tau) \geq 0.6$. \qedhere
  \end{enumerate}
\end{proof}

We conclude this section with lower bounds on~$r$ or~$\lambda_1$
at~$\gamma_k\tau$ for~$\tau\in \Fund'$ and~$1\leq k\leq 3$.

\begin{lem}
  \label{lem:lambda1-gmai}
  For every $\tau\in\Fund'$, we have
  \begin{displaymath}
    r(\gamma_1\tau) \geq \dfrac{9\, y_1(\tau)}{34\abs{z_1(\tau)}^2},
    \quad r(\gamma_2\tau) \geq \dfrac{9\, y_2(\tau)}{34\abs{z_2(\tau)}^2},
    \quad\text{and}\quad \lambda_1(\gamma_3\tau) \geq \dfrac{9}{44\,y_2(\tau)}.
  \end{displaymath}
\end{lem}

\begin{proof}
  We have
  \begin{displaymath}
    \im(\gamma_1\tau) = \mat{z_1}{z_3}{0}{-1}^{-t}\im(\tau)
    \mat{\conj{z}_1}{\conj{z}_3}{0}{-1}^{-1}
    = \frac{1}{\abs{z_1}^2}\mat{y_1}{\alpha}{\alpha}{\beta}
  \end{displaymath}
  with $\alpha = y_1 x_3 - y_3 x_1$, so
  $\abs{\alpha}\leq \frac{3}{4}y_1$. Moreover,
  \begin{displaymath}
    \det\im(\gamma_1\tau) = \frac{1}{\abs{z_1}^2}\det\im(\tau)
  \end{displaymath}
  and $\det \im(\tau)\geq 9/16$, so
  \begin{align*}
    \beta&\leq \dfrac{\abs{z_1}^2}{y_1} \det\im(\tau) + \dfrac{9}{16}y_1
           \quad\text{and}\quad
    \beta \geq \frac{\abs{z_1}^2}{y_1} \det\im(\tau) \geq \frac{9}{16}y_1.
  \end{align*}
  Therefore,
  \begin{displaymath}
    \lambda_1(\gamma_1\tau)\geq \dfrac{\det \im(\gamma_1\tau)}{\Tr\im(\gamma_1\tau)}
    \geq
    \frac{y_1}{\abs{z_1}^2}
    \frac{1}{1+\frac{25}{16} \frac{y_1^2}{\abs{z_1}^2\det\im(\tau)}}
    \geq \frac{9 y_1}{34\abs{z_1}^2}.
  \end{displaymath}
  We did not use the property that $y_1\leq y_2$, so the same proof
  works for~$\gamma_2\tau$.  Finally, we consider $\gamma_3\tau$. We
  have
  \begin{displaymath}
    \im(\gamma_3\tau) = \dfrac{1}{\abs{\det\tau}^2} \mat{\beta_1}{\alpha}{\alpha}{\beta_2}
  \end{displaymath}
  with
  \begin{align*}
    \beta_1 &= y_1\abs{z_3}^2 + y_2\abs{z_1}^2 -
    y_3(z_1 \conj{z}_3 + z_3 \conj{z}_1),\\
    \quad \beta_2 &=  y_1\abs{z_2}^2 + y_2\abs{z_3}^2 -
    y_3(z_2 \conj{z}_3 + z_3 \conj{z}_2).
  \end{align*}
  We compute
  \begin{align*}
    \abs{\det\tau}^2 \Tr\im(\gamma_3\tau) = \beta_1+\beta_2
    &\leq 
    y_1y_2^2 + y_1^2 y_2 + \frac{1}{2}(y_1+y_2+\abs{y_3}) \leq \frac{11}{3}y_1y_2^2
  \end{align*}
  because $y_1y_2\geq 3/4$. Therefore,
  \begin{displaymath}
    \lambda_1(\gamma_3\tau) \geq \dfrac{3\det \im(\tau)}{11 y_1y_2^2} \geq \frac{9}{44y_2}.
    \qedhere
  \end{displaymath}
\end{proof}

\section{Proof of the main theorem}
\label{sec:proof}

In this final section, we prove \cref{thm:main} by separating
different cases according to the value of~$n$. If~$n$ is large enough,
then \cref{lem:goodpos,lem:lambda1-gmai} are enough to conclude;
if~$n$ is smaller, then we apply the theta transformation formula
(\cref{prop:transf}) and the bounds on other theta constants given
in~§\ref{sec:bounds}.

In the proofs, we use numerical calculations, typically in order to
show that a given angle~$\alpha(q)$ is smaller than~$\pi/2$ for
certain values of~$q$. Such calculations are easily certified using
interval arithmetic, since the functions~$\alpha(q)$ we consider are
simple: they are either increasing or convex functions of~$q$.

In order to help the reader visualize the estimates, we created
pictures using GeoGebra \cite{hohenwarter_GeoGebra5882020}.

\begin{prop}
  \label{prop:gamma0}
  Let $\tau\in\Fund'$. Then for every $n\geq 0$, the theta constants
  $\theta_j(2^n\tau)$ for $0\leq j\leq 3$ are in good position.
\end{prop}

\begin{proof}
  For every $n\geq 0$, we have
  \begin{displaymath}
    r(2^n\tau) = 2^n r(\tau) \geq \sqrt{3}/4 \geq 0.4,
  \end{displaymath}
  so the result follows from \cref{lem:goodpos}.
\end{proof}

\begin{lem} \label{lem:goodpos-46} Let $\tau\in\Fund'$.
  \begin{enumerate}
  \item For every $n\geq 0$ such that $2^n\leq 8.77 y_1(\tau)$, the
    theta constants $\theta_j(\tau_1^{(n)})$ for $j\in\{0,2,4,6\}$ are
    in good position.
  \item For every $n\geq 0$ such that $2^n\leq 8.77 y_2(\tau)$, the
    theta constants $\theta_j(\tau_2^{(n)})$ for $j\in\{0,1,8,9\}$ are
    in good position.
  \end{enumerate}
\end{lem}

\begin{proof}
  We only prove the first statement, the second one being
  symmetric. We separate three cases: $n=0$, $n=1$, and $n\geq 2$.

  \paragraph{Case 1:} $n=0$. Then $\tau_1^{(n)} = \tau$. By
  \cite[Prop.~7.7]{streng_ComplexMultiplicationAbelian2010}, we have
  \begin{align*}
    \abs{\theta_j(\tau) - 1} \leq 0.405 &\qquad \text{for } j\in\{0,1,2,3\}, \text{ and} \\
    \abs{\frac{\theta_j(\tau)}{\xi_{4,6}(\tau)}-1} \leq 0.348 &\qquad \text{for } j\in\{4,6\}.
  \end{align*}
  The absolure value of the argument of $\xi_{4,6}(\tau)$ is at
  most~$\pi/8$. Therefore the angle between any two $\theta_j(\tau)$
  for $j\in\{0,1,2,3,4,6\}$ is at most
  \begin{displaymath}
    \frac{\pi}{8} + \arcsin(0.348) + \arcsin(0.405) < \frac\pi2.
  \end{displaymath}
  % The geometric situation can be pictured as follows.
  % \addpic{theta46.eps}

  \paragraph{Case 2:} $n=1$.  We study the relative positions of
  $\xi_{0,2}$ and $\xi_{4,6}$ at~$\tau_1^{(1)}$.  Since
  $\abs{2^{-n}x_1(\tau)}\leq 1/4$, the absolute value of the argument
  of $\xi_{4,6}(\tau_1^{(1)})$ is bounded above by~$\pi/16$. Moreover,
  \begin{displaymath}
    \babs{\xi_{0,2}(\tau_1^{(1)})}\geq 1, \quad
    \babs{\arg(\xi_{0,2}(\tau_1^{(1)}))}
    \leq \arctan\Bigl(\frac{2q_1 \sin(\pi/4)}{1+2q_1\cos(\pi/4)}\Bigr),
  \end{displaymath}
  and the arguments of $\xi_{0,2}$ and $\xi_{4,6}$ have the same
  sign. Therefore the angle between any two $\theta_{j}(\tau_1^{(1)})$
  for $j\in \{0,2,4,6\}$ is at most
  \begin{displaymath}
    \max\Bigl\{\frac\pi{16},\arctan\Bigl(\frac{2q_1 \sin(\pi/4)}{1+2q_1\cos(\pi/4)}\Bigr)\Bigr\}
    + \arcsin\rho_{4,6}^{(4)}(q_1,q_2) + \arcsin (\rho_0(q_1,q_2) + 2q_2)
  \end{displaymath}
  by \cref{lem:bound46,lem:bound012}.  This quantity is less than
  $\pi/2$ because
  \begin{displaymath}
    q_2(\tau_1^{(1)})\leq \exp(-\pi\sqrt{3}) \quad\text{and}\quad
    q_1(\tau_1^{(1)})\leq \exp(-\pi\sqrt{3}/8).
  \end{displaymath}
  % the picture is as follows.
  % \addpic{theta46-1.eps}

  \paragraph{Case 3:} $n\geq 2$. We proceed as in Case 2, but we now
  have
  \begin{displaymath}
    q_2(\tau_1^{(n)})\leq \exp(-2\pi\sqrt{3}),\quad
    8\babs{y_3(\tau_1^{(n)})}\leq y_2(\tau_1^{(n)}), \quad\text{and}\quad
    \babs{x_1(\tau_1^{(n)})} \leq \frac18.
  \end{displaymath}
  Therefore the angle between the
  $\theta_j(\tau_1^{(n)})$ for $j\in\{0,2,4,6\}$ is bounded by
  \begin{equation}
    \label{eq:angle}
    \begin{aligned}
      &\max\Bigl\{\frac\pi{32},
      \arctan\Bigl(\frac{2 q_1 \sin(\pi/8)}{1+2 q_1 \cos(\pi/8)}\Bigr)\Bigr\}\\
      &+ \arcsin (\rho_{0}(q_1,\exp(-2\pi\sqrt{3})) + 2\exp(-2\pi\sqrt{3}))\\
      &+ \arcsin \rho_{4,6}^{(8)}(q_1,\exp(-2\pi\sqrt{3})).
    \end{aligned}
  \end{equation}
  This angle remains less that $\pi/2$ when
  $q_1(\tau_1^{(n)})\leq 0.699$. This is the case when
  $2^{n}\geq 8.77y_1(\tau)$.
\end{proof}

The geometric situation in Case~3 of \cref{lem:goodpos-46} can be represented as follows.
\addpic{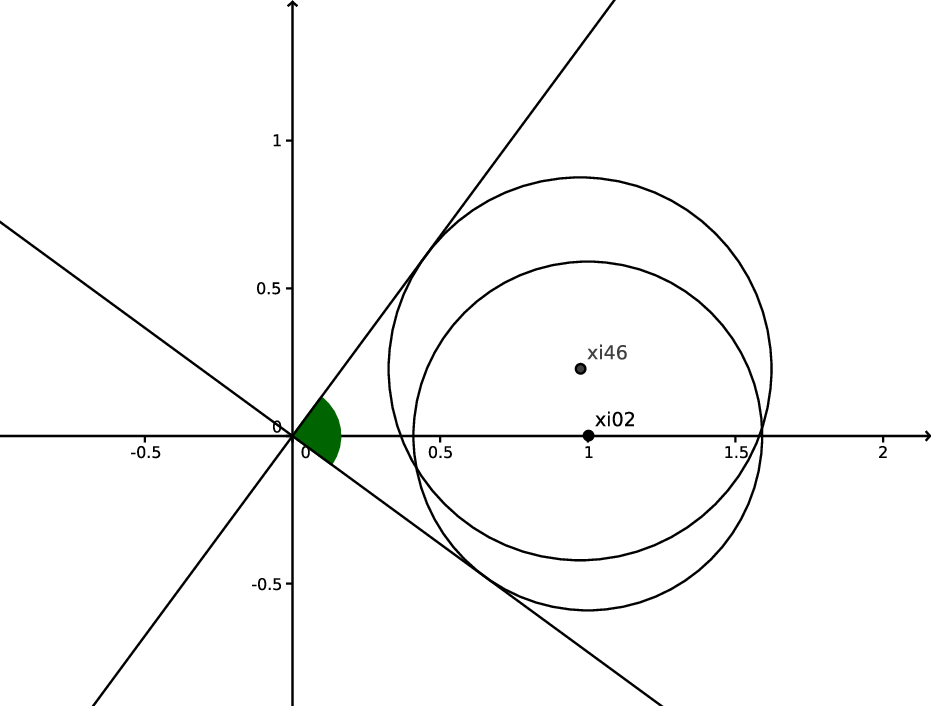}

In this picture, we take~$q_1=0.699$, and represent two complex
numbers~$\xi_{0,2}$ and~$\xi_{4,6}$ with modulus one, separated by
an angle of
\begin{displaymath}
  \max\Bigl\{\frac\pi{32},
  \arctan\Bigl(\frac{2 q_1 \sin(\pi/8)}{1+2 q_1 \cos(\pi/8)}\Bigr)\Bigr\} \simeq 0.22.
\end{displaymath}
Then we draw disks centered in~$\xi_{0,2}$ and~$\xi_{4,6}$ with
radii~$\rho_0(q_1,\exp(-2\pi\sqrt{3}))$
and~$\rho_{4,6}^{(8)}(q_1, \exp(-2\sqrt{3}))$ respectively. Finally
we represent the smallest angular sector seen from the origin
containing these two disks. The green angle is equal to the
quantity~\eqref{eq:angle}, and is indeed smaller than~$\pi/2$.

\begin{prop}
  \label{prop:gamma12}
  Let $\tau\in\Fund'$.
  \begin{enumerate}
  \item For every $n\geq 0$, the theta constants
    $(\theta_j(2^n\gamma_1\tau))_{0\leq j\leq 3}$ are in good
    position.
  \item For every $n\geq 0$, the theta constants
    $(\theta_j(2^n\gamma_2\tau))_{0\leq j\leq 3}$ are in good
    position.
  \end{enumerate}
\end{prop}

\begin{proof}
  By \cref{lem:lambda1-gmai}, we have
  \begin{displaymath}
    r(\gamma_1\tau) \geq \frac{9\, y_1}{34\abs{z_1}^2} \geq
    \frac{9\, y_1}{34(1/4 + y_1^2)} \geq \frac{0.205}{y_1(\tau)}
  \end{displaymath}
  because $y_1(\tau)\geq \sqrt{3}/2$. By \cref{lem:goodpos}, the
  $\theta_j(2^n\gamma_1\tau)$ for $0\leq j\leq 3$ are in good position
  when $2^n r(\gamma_1\tau) \geq 0.4$. This is the case when
  $2^n \geq 1.96 y_1$. On the other hand, \cref{lem:goodpos-46}
  applies when $2^n\leq 8.77y_1$. The second statement is proved in
  the same way.
\end{proof}

\begin{lem}
  \label{lem:tau3n}
  Let $\tau\in\Fund'$. Then, for every $n\geq 0$ such that
  $2^n\leq 1.66 y_1$,
  the theta constants $\theta_j(\tau_3^{(n)})$ for $j\in\{0,4,8,12\}$
  are in good position.
\end{lem}

\begin{proof}
  Write $q = q_1(\tau_3^{(n)})$ for short.  We separate two cases:
  $n\geq 1$, and $n=0$.

  \paragraph{Case 1:} $n\geq 1$. In this case, we have
  \begin{displaymath}
    \babs{x_j(\tau_3^{(n)})}\leq 1/4 \quad \text{for each } 1\leq j\leq 3.
  \end{displaymath}
  Therefore, given the expressions
  of~$\xi_0$,~$\xi_{4,6}$,~$\xi_{8,9}$ and~$\xi_{12}$
  (see~\eqref{eq:xi}), and by
  \cref{lem:bound46,lem:bound89,lem:bound012,lem:bound12},
  \begin{itemize}
  \item The angle between $\theta_4(\tau_3^{(n)})$ and
    $\theta_8(\tau_3^{(n)})$ is bounded by
    \begin{displaymath}
      \frac{\pi}{8} + 2\arcsin\rho_{4,6}^{(2)}(q, q).
    \end{displaymath}
  \item The angle between $\theta_4(\tau_3^{(n)})$ (or $\theta_8$) and
    $\theta_0(\tau_3^{(n)})$ is bounded by
    \begin{displaymath}
      \frac\pi{16} + \arcsin\rho_{4,6}^{(2)}(q,q) +  2q\sin(\pi/4) + \arcsin \rho_0(q, q).
    \end{displaymath}
  \item The angle between $\theta_{12}(\tau_3^{(n)})$ and $\theta_4(\tau_3^{(n)})$ (or
    $\theta_8$) is bounded by
    \begin{displaymath}
      \frac{3\pi}{16} + \arcsin \rho_{12}(q,q) + \arcsin \rho_{4,6}^{(2)}(q,q).
    \end{displaymath}
  \item The angle between $\theta_{12}(\tau_3^{(n)})$ and $\theta_0(\tau_3^{(n)})$ is bounded by
    \begin{displaymath}
      \frac{\pi}{4} + \arcsin\rho_{12}(q,q) + \arcsin \rho_0(q,q).
    \end{displaymath}
  \end{itemize}
  All these quantities remain less than $\pi/2$ when $q\leq
  0.151$. This is the case when $2^n\leq 1.66y_1$.%  Geometrically,
  % these estimations can be represented as follows.
  
  % \addpic{theta12-3.eps} \addpic{theta12-5.eps} \addpic{theta12-4.eps}

  \paragraph{Case 2:} $n=0$. In this case, we have
  $q\leq \exp(-\pi\sqrt{3}/2)$. Therefore,
  \begin{itemize}
  \item The angle between $\theta_4$ and $\theta_8$ is bounded by
    \begin{displaymath}
      \frac\pi4 + 2\arcsin\rho_{4,6}^{(2)}(q,q) < \frac\pi 2.
    \end{displaymath}
  \item The angle between $\theta_4$ (or $\theta_8$) and $\theta_0$ is
    bounded by
    \begin{displaymath}
      \frac\pi 8 + \arcsin \rho_{4,6}^{(2)}(q,q) + \arcsin(\rho_0(q,q)+4q) < \frac\pi2.
    \end{displaymath}
  \item The angle between $\theta_{12}$ and $\theta_4$ (or $\theta_8$)
    is bounded by
    \begin{displaymath}
      \frac{3\pi}{8} + \arcsin\rho_{12}(q,q) + \arcsin \rho_{4,6}^{(2)}(q,q) < \frac\pi2.
    \end{displaymath}
  \end{itemize}
  These estimations can be represented as follows, with similar
  conventions as in the picture after \cref{lem:goodpos-46}:
  \addpic{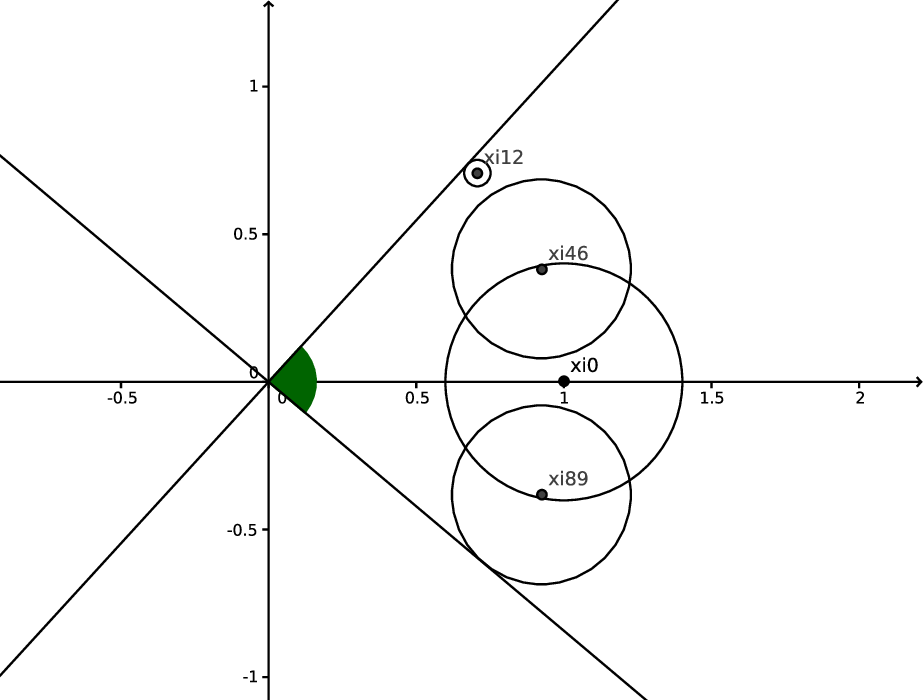}

  We finally study the angle between $\theta_{12}$ and $\theta_0$.
  The argument of $\xi_{12}(\tau)$ is $x_1/4+x_2/4+\beta$ with
  $\beta = \arg(\exp(i\pi z_3/2) + \exp(-i\pi z_3/2))$. Up to
  conjugating, we may assume that $y_3\geq 0$ and $x_3\geq 0$. Then
  \begin{align*}
    \exp(i\pi z_3/2) + \exp(-i\pi z_3/2) = \exp(-i\pi z_3/2)(1 + \exp(i\pi z_3))
  \end{align*}
  so
  \begin{align*}
    \beta + \frac{\pi x_3}{2} \geq 
    \arctan\left(\frac{q_3\sin(\pi x_3)}{1+ q_3}\right)
    \geq \arctan\left(\frac{2x_3q_3}{1+q_3}\right).
  \end{align*}
  In general, we have
  \begin{displaymath}
    \abs{\beta} 
    \leq \frac{\pi}{4} - \arctan\left(\frac{q^{1/2}}{1+q^{1/2}}\right).
  \end{displaymath}
  On the other hand,
  \begin{displaymath}
    1\leq \re(\xi_0(\tau))\leq 1+4q,\quad \im(\xi_0(\tau)) = 2q_1\sin(\pi x_1) + 2 q_2\sin(\pi x_2).
  \end{displaymath}
  We discuss two cases according to the signs of $x_1$ and $x_2$:
  \begin{itemize}
  \item If $x_1$ and $x_2$ have opposite signs, then the angle between
    $\theta_{12}$ and $\theta_0$ is at most
    \begin{displaymath}
      \frac{3\pi}{8}
      + \arctan(2q) + \arcsin \rho_{12}(q,q) + \arcsin\rho_0(q,q).
    \end{displaymath}
  \item If $x_1$ and $x_2$ have the same sign, say positive, then
    \begin{displaymath}
      \frac{x_1+x_2}{4} - \arg\xi_0(\tau) \leq \frac{x_1+x_2}{4}.
    \end{displaymath}
    Therefore the angle between $\theta_{12}$ and $\theta_0$ is at most
    \begin{displaymath}
      \frac\pi2 -  \arctan\left(\frac{q^{1/2}}{1+q^{1/2}}\right) + \arcsin\rho_{12}(q,q) + \arcsin\rho_{0}(q,q).
    \end{displaymath}
    This function of $q$ is not increasing, but it is convex.
  \end{itemize}
  A numerical investigation shows that both quantities remain less
  than $\pi/2$ when $q\leq \exp(\sqrt{3}/2)$. %theta12.eps, theta12-7.eps
\end{proof}

\begin{lem}
  \label{lem:tau4n}
  Let $\tau\in\Fund'$, and let $n_0\in\N$ such that
  $2^{n_0} > 1.66 y_1$. Then, for every $n\geq n_0$ such that
  $2^n\leq 4.2y_2(\tau)$, the theta constants $\theta_j(\tau_4^{(n)})$
  for $j\in\{0,1,8,9\}$ are in good position.
\end{lem}

\begin{proof}
  By assumption, we have
  $y_1(\tau_4^{(n)}) \geq \frac{3}{4}\cdot 1.66 \geq 1.24$, so
  $q_1(\tau_4^{(n)})\leq 0.021$. Moreover we must have $n\geq 1$, so
  by \cref{lem:tau4}, $\babs{x_2(\tau_4^{(n)})}\leq 9/16$, and
  \begin{displaymath}
    \babs{y_3(\tau_4^{(n)})} \leq \frac38 y_1(\tau_4^{(n)}).
  \end{displaymath}
  Therefore, we can apply \cref{lem:bound0prime} with $k=8/3$: we have
  \begin{align*}
    \babs{\theta_j(\tau_4^{(n)}) - \xi_{0,1}(\tau_4^{(n)})}
    &\leq \rho_{0,1}'^{(8/3)}(0.021, q_2(\tau_4^{n}))
      \quad \text{for } j\in\{0,1\},\\
    \abs{\frac{\theta_j(\tau_4^{(n)})}{\xi_{8,9}(\tau_4^{(n)})} -1}
    &\leq \rho_{8,9}'^{(8/3)}(0.021, q_2(\tau_4^{(n)})
      \quad \text{for }j\in\{8,9\}.
  \end{align*}
  Let us investigate the difference between the arguments
  of~$\xi_{8,9}(\tau_4^{(n)})$ and~$\xi_{0,1}(\tau_4^{(n)})$. Both
  have the sign of~$x_2(\tau_4^{(n)})$, which we may assume to be
  positive. If the argument of~$\xi_{8,9}$ is the largest, then the
  difference is bounded by
  \begin{displaymath}
    \arg \xi_{8,9}(\tau_4^{(n)}) \leq \frac{9\pi}{64}.
  \end{displaymath}
  If the argument of $\xi_{0,1}$ is the largest, we distinguish two
  cases. If $x_2(\tau_4^{(n)})\geq \frac{3\pi}{8}$, then
  \begin{displaymath}
    \arg\xi_{0,1}(\tau_4^{(n)}) - \arg \xi_{8,9}(\tau_4^{(n)}) \leq 
    \arctan\Bigl(\frac{2q_2}{1 + 2q_2\cos(9\pi/16)}\Bigr) - \frac{3\pi}{32}.
  \end{displaymath}
  On the other hand, if $x_2(\tau)\leq 3\pi/8$, then
  \begin{displaymath}
    \arg\xi_{0,1}(\tau_4^{(n)}) - \arg \xi_{8,9}(\tau_4^{(n)})\leq \arg \xi_{0,1}(\tau_4^{(n)})
    \leq \arctan\Bigl(\frac{2q_2\sin(3\pi/8)}{1 + 2q_2\cos(3\pi/8)}\Bigr)
  \end{displaymath}
  Note that $\babs{\xi_{0,1}(\tau_4^{(n)})}$ is always greater than
  $\cos(\pi/16)$. Therefore the angle between the
  $\theta_j(\tau_4^{(n)})$ for $j\in\{0,1,8,9\}$ is at most
  \begin{align*}
    &\max\left\{\frac{9\pi}{64},\
      \arctan\Bigl(\frac{2q_2}{1 + 2q_2\cos(9\pi/16)}\Bigr) - \frac{3\pi}{32},\
      \arctan\Bigl(\frac{2q_2\sin(3\pi/8)}{1 + 2q_2\cos(3\pi/8)}\Bigr)\right\}\\
    &\qquad+ \arcsin \rho_{8,9}'^{(8/3)}(0.021,q_2)
      + \arcsin \frac{\rho_{0,1}'^{(8/3)}(0.021,q_2)}{\cos(\pi/16)}.
  \end{align*}
  This quantity is less than $\pi/2$ when
  $q_2(\tau_4^{(n)})\leq 0.38$. Since
  $y_2(\tau_4^{(n)})\geq \frac{3}{2^{n+2}} y_2(\tau)$ by
  \cref{lem:tau4}, this is the case when $2^n \leq 2.43
  y_2(\tau)$. % The geometric situation is represented on the following
  % picture.
  % \addpic{tau4n-corr.eps}

  On the other hand, if $2^n > 2.43 y_2(\tau)$, then we must
  have $n\geq 2$. Moreover,
  \begin{displaymath}
    y_1(\tau_4^{(n)}) > 2.43 \frac{y_1(\tau) y_2(\tau)}{\abs{z_1(\tau)^2}} > 1.82,
  \end{displaymath}
  so $q_1(\tau_4^{(n)}) < 0.0033$. Then, the angle bound improves to
  \begin{align*}
    &\max\left\{\frac{9\pi}{128},
    \arctan\Bigl(\frac{2q_2 \sin(9\pi/32)}{1 + 2q_2\cos(9\pi/32)}\Bigr)\right\}\\
    &\qquad
    + \arcsin \rho_{8,9}'^{(16/3)}(0.0033,q_2) + \arcsin \rho_{0,1}'^{(16/3)}(0.0033,q_2).
  \end{align*}
  This quantity is less than $\pi/2$ when
  $q_2(\tau_4^{(n)})\leq 0.571$, and the latter inequality holds when
  $2^n\leq 4.2 y_2(\tau)$.
\end{proof}

\begin{prop}
  \label{prop:gamma3}
  Let $\tau\in\Fund'$. Then, for
  every $n\geq 0$, the theta constants $\theta_j(2^n\gamma_3\tau)$ for
  $0\leq j\leq 3$ are in good position.
\end{prop}

\begin{proof}
  By \cref{lem:lambda1-gmai}, we have
  \begin{displaymath}
    \lambda_1(\gamma_3\tau) \geq \frac{9}{44 y_2(\tau)}.
  \end{displaymath}
  Therefore, by \cref{lem:goodpos}, the theta constants are in good position as soon as
  \begin{displaymath}
    2^n \frac{9}{44 y_2(\tau)} \geq 0.6, \quad \text{or} \quad 2^n\geq 2.94 y_2(\tau).
  \end{displaymath}
  When $n$ is smaller, we use the transformation
  formulas. \Cref{lem:tau3n} applies when $2^n\leq 1.66 y_1(\tau)$,
  and \cref{lem:tau4n} applies when
  $1.66 y_1(\tau) < 2^n \leq 4.2 y_2(\tau)$.
\end{proof}

\Cref{prop:gamma0,prop:gamma12,prop:gamma3} together imply \cref{thm:main}.

\bibliographystyle{abbrv}
\bibliography{dupont.bib}

\bigskip

{\it  \noindent
Jean Kieffer \\
Institut de Mathématiques de Bordeaux \\ 
351 cours de la Libération, 33400 Talence, France
}
\end{document}